\documentclass[review]{elsarticle}

\usepackage{amsmath}
\usepackage{amsthm}
\usepackage{amsfonts}
\usepackage{amssymb}
\usepackage{graphicx}
\usepackage{tikz}
\usepackage{amscd}
\usepackage{framed}
\graphicspath{ {images/} }
\newtheorem{Theorem}{Theorem}[section]

\newtheorem{Lemma}[Theorem]{Lemma}
\newtheorem{Corollary}[Theorem]{Corollary}
\theoremstyle{definition}

\theoremstyle{remark}

\numberwithin{equation}{section}

\newcommand{\R}{{\mathbb R}}
\newcommand{\C}{{\mathbb C}}
\newcommand{\N}{{\mathbb N}}

\usepackage{lineno,hyperref}
\modulolinenumbers[5]

\journal{Journal of Differential Equations}

\begin{document}

\begin{frontmatter}
\title{The $m$-functions of discrete Schr\"odinger operators are sparse compared to those for Jacobi operators}
\author{Injo Hur\fnref{myfootnote}}
\address{Department of Mathematical Sciences\\
Ulsan National Institute of Science and Technology\\
Ulsan, Republic of Korea}

\begin{abstract}
We explore the sparsity of Weyl-Titchmarsh $m$-functions of discrete Schr\"odinger operators. Due to this, the set of their $m$-functions cannot be dense on the set of those for Jacobi operators. All this reveals why an inverse spectral theory for discrete Schr\"odinger operators via their spectral measures should be difficult. To obtain the result, de Branges theory of canonical systems is applied to work on them, instead of Weyl-Titchmarsh $m$-functions.  
\end{abstract}


\end{frontmatter}

%

\section{Introduction}\label{secintro}
Inverse spectral theories enable us to comprehend operators through their spectral data. Several results have been well-known in one-dimensional setting via so-called spectral measures or Weyl-Titchmarsh $m$-functions: there are one-to-one correspondence between canonical systems and Herglotz functions \cite{deB, Win}, another correspondence between Jacobi operators and probability measures \cite{RemLN,Teschl}, and the representations of the spectral measures for Schr\"odinger operators \cite{GL,GS2,Mar2,RemdB,SimIST}. (See precise definitions in the relevant sections below.) 

Roughly speaking, the first two turn out to be very easy to apply. This is because both spaces of all Herglotz functions (or equivalently, canonical systems due to the 1-1 correspondence above) and all probability measures (or Jacobi operators) whose supports are in some bounded and closed interval (for example, $[-2,2]$) are known to be compact. In other words, these two spaces are so ``affluent"  to not miss any related spectral measures. 

Unlike this, some difficult statements are necessary in order to describe the spectral measures of Schr\"odinger operators \cite{GL,GS2,Mar2,RemdB,SimIST}. These have some Gelfand-Levitan type conditions, which are presented in terms of a Fourier-Laplace type transform of their spectral measures. Since 
not every Herglotz function can occur as an $m$-function of some Schr\"odinger operator, i.e., there are lots of missing pieces on the space of Schr\"odinger operators compared to canonical systems, the compactness cannot be preserved on this space. 

Fortunately, in spite of this insufficiency, this set seems to be still ample in order to formulate an inverse spectral theory. One of the aspects showing this is the density of all their $m$-functions on the (compact) space of the Herglotz functions \cite{Hur3}. Schr\"odinger operators are enough to approximate any canonical systems.

In this paper, we would like to reveal the sparsity of discrete Schr\"odinger operators, which is contrary to the idea   based on the density outcome of Schr\"odinger operators \cite{Hur3}. (In many applications we would expect to have the same conclusion on both continuous and discrete settings.) Due to the scarcity, it will be shown that there is a Jacobi operator whose $m$-function is far from those of discrete Schr\"odinger operators.  This non-density seems to advocate the fact that no inverse spectral theory does appear to be known for discrete Schr\"odinger operators. They are too scanty to be described. See Table 1 below. 

\begin{table}[h]
\caption{Inverse spectral theories in one-dimensional space}
\centering
\begin{tabular}{c c}
\hline\hline 
Continuous &  Discrete \\ [0.5ex]
\hline 
Canonical systems: compact & Jacobi operators: compact \quad \\
$\implies$ Easy theory & \quad $\implies$  Easy theory \\

Schr\"odinger operators: dense &  \textrm{ } Discrete Schr\"odinger operators: scarcity\\
$\implies$ Difficult theory &  $\implies$ ``Hard to get such a theory"\\[0.5ex]
\hline
\end{tabular}
\end{table}
  
To derive the sparsity or non-density, the fundamental idea in \cite{Hur3} is followed: instead of dealing with spectral measures or $m$-functions directly, let's work on the related canonical systems. For this, so-called de Branges theory of canonical systems \cite{Ach,deB,HSW,KL,RemdB,Win3,Win4,Win,Win2} has been applied in \cite{Hur3}. 

Similarly in this paper,  we construct a route for discrete Schr\"odinger operators via canonical systems. More precisely, Theorem \ref{classificationofjcs} and Theorem \ref{classificationofdscs} show the complete expressions of the canonical systems related to Jacobi and discrete Schr\"odinger operators, respectively. Let us call them \emph{Jacobi canonical systems} and \emph{discrete Schr\"odinger canonical systems} for convenience. Since no inverse theory of discrete Schr\"odinger operators has been known, it is Theorem \ref{classificationofdscs} that seems to be the only way to deal with them. 

As an application of Theorem \ref{classificationofdscs}, some Jacobi canonical system will be constructed, such that it cannot be approximated by any sequences of discrete Schr\"odinger canonical systems. By two expressions above, this implies that there is a Jacobi operator whose $m$-function is far from those of discrete Schr\"odinger operators. See Theorem \ref{nodensity} below.

Let us enlighten the usefulness of Theorem \ref{classificationofdscs}. It can be interpreted as an inverse spectral theory for discrete Schr\"odinger operators in the version of (trace-normed) canonical systems. Note  again that there has been no such a  theory for discrete Schr\"odinger operators via their $m$-functions, or equivalently, their spectral measures. This theorem will therefore be very useful to deal with unsolved problems, especially for discrete Schr\"odinger operators, in the field of the spectral theory.

This paper is organized as follows. Basic materials, such as Jacobi operators, canonical systems and their $m$-functions, are provided on the following section. In Section \ref{classification} we classify the canonical systems that are the eigenvalue equations for Jacobi or discrete Schr\"odinger operators in disguise. With this classification, we, in Section \ref{example}, then construct a canonical system corresponding to some Jacobi operator whose $m$-function cannot be approachable by the $m$-functions of discrete Schr\"odinger operators. 

\section{Preliminaries}\label{secbasicm}
\subsection{Jacobi and discrete Schr\"odinger operators}\label{Jacobi}
A Jacobi operator $\mathcal{J}$ is a difference operator on $\ell^2(\N)$, defined by
\begin{equation*}
(\mathcal{J}u)_n=
\begin{cases}
\textrm{ } a_{n-1}u_{n-1}+a_n u_{n+1}+b_n u_n, \quad n\ge2 \\
\textrm{ } a_n u_{n+1}+b_n u_n, \quad n=1
\end{cases}
\end{equation*}
where $a, b\in \ell^{\infty}({\N})$, $a_n<0$, $b_n\in\R$. Many other papers assume that $a_n>0$, but it is well-known that the signs of $a_n$ can be switched such that two Jacobi operators are unitarily equivalent. Discrete Schr\"odinger operators are the Jacobi operators satisfying $a_n= -1$ for all $n$. 

Then a Weyl-Titchmarsh $m$-function corresponding to the given Jacobi operator is defined by
\begin{equation}\label{mfnforjacobi}
m_{\mathcal{J}}(z) = - \frac{\tilde{y}(1,z)}{a_0 \tilde{y}(0,z)},
\end{equation}
where $\tilde{y}(\cdot, z)$ is a solution to the formal equation 
\begin{equation}\label{jacobieqn}
a_{n-1}y_{n-1}+a_n y_{n+1}+b_ny_n=zy_n, \quad n\ge 1,
\end{equation} 
such that it  is square-summable near $\infty$. The condition that $\mathcal{J}$ is bounded makes sure that $\tilde{y}$ is unique up to multiplicative constants, which implies that  (\ref{mfnforjacobi}) is well-defined.  
In this paper we assume that $a_0=-1$. This is practicable, since the Dirichlet boundary condition, i.e., $y_0=0$, is the considered one at 0, when the operator is specified. 

Weyl $m$-functions for Jacobi operators $m_{\mathcal{J}}$ are Herglotz functions, that is, they map the upper half plane $\mathbb{C}^+$  holomorphically to itself. With the Herglotz representation (which is similar to Poisson integral formula for positive harmonic functions), 
it is well-known that 
$m_{\mathcal{J}}$'s are expressed by  
\begin{equation}\label{HRofJ}
m_{\mathcal{J}}(z)=\int_{-\infty}^{\infty} \frac{d\rho_{\mathcal{J}}(t)}{t-z}.
\end{equation}
Here the spectral measures $d\rho_{\mathcal{J}}$ are probability measures whose supports are bounded, but not finite. 

It, moreover, turns out that the converse of the fact above is also true. All this means that Jacobi operators have an very easy inverse spectral theory  (Theorem 13.9 in \cite{RemLN}): there is one-to-one correspondence between (bounded) Jacobi operators and probability measures whose supports are infinite and bounded in $\R$.  
Put differently, the Cauchy transforms of such probability measures are the $m$-functions of Jacobi operators.

As opposed to this, there is no well-known result to tell us which Herglotz functions are the $m$-functions for discrete Schr\"odinger operators. This difficulty is the main reason why we cannot work on them directly. See \cite{RemLN} or \cite{Teschl} for all these properties above.

\subsection{Canonical systems}

\subsubsection{Canonical systems and de Branges theory} 
To see more larger picture between equations and Herglotz functions let us consider a (half-line) canonical system,  
\begin{equation}
\label{cs}
J{\bf{u}}'(x,z)=zH(x){\bf{u}}(x,z), \quad x\in(0,\infty),
\end{equation}
 where $J=\big( \begin{smallmatrix}  0 & -1 \\ 1 &0 \end{smallmatrix}$\big) and $H$ (called a Hamiltonian) is a positive semidefinite $2\times2$ matrix whose entries are real-valued, locally integrable functions. Here $z$ is a spectral parameter. In particular, a canonical system or a Hamiltonian $H$ is called \textit{trace-normed}, if $\textrm{Tr }H(x)=1$ for almost all $x$ in $(0,\infty)$. For (\ref{cs}) we always place a boundary condition at 0,
\begin{equation}
\label{bcat0forcs}
u^1(0,z)=0,
\end{equation}
where $u^1$ is the first component of ${\bf{u}}=\big( \begin{smallmatrix}  u^1 \\ u^2\end{smallmatrix}$\big). 

Similar to $m_{\mathcal{J}}$, a Weyl $m$-function $m_H$ corresponding to the given canonical system can be expressed by 
\begin{equation}
\label{mfnforcs}
m_H(z)=\frac{\tilde{u}^2(0,z)}{\tilde{u}^1(0,z)},
\end{equation}
where $\tilde{{\bf{u}}}=\big( \begin{smallmatrix} \tilde{u}^1 \\ \tilde{u}^2 \end {smallmatrix} \big) $ is a solution to (\ref{cs}) satisfying   
\begin{equation}
\label{H-int}
\int_0^{\infty} \tilde{{\bf{u}}}^*(x,z)H(x)\tilde{{\bf{u}}}(x,z)dx<\infty.
\end{equation}
Here $^*$ means the Hermitian adjoint. Such a solution satisfying (\ref{H-int}) is called \textit{$H$-integrable}. Note that a half-line trace-normed canonical system is always in a limit-point case at $\infty$, which was obtained from the original argument by \cite{deB} or an alternative proof in \cite{Ach}. In other words, there is only one $H$-integrable solution up to multiplicative constants, and therefore (\ref{mfnforcs}) is well-defined. See \cite{Hur3,Win,Win2} for more details.

De Branges \cite{deB} and Winkler \cite{Win} then showed that, for a given Herglotz function, there exists a unique half-line trace-normed canonical system with (\ref{bcat0forcs}), such that its $m$-function $m_H$ is the given Herglotz function. In \cite{Hur3} it was shown that it is much easier to see which canonical system is an eigenvalue equation of some Schr\"odinger operator, than which Herglotz function is an $m$-function of some Schr\"odinger operator. Based on this theme, this one-to-one correspondence will be essential later in order to cope with canonical systems.

\subsubsection{Singular intervals}\label{singularinterval}
For the purpose of the later use, we need  the following treatment, based on \cite{deB} or mainly Section 10 in \cite{RemdB}. An interval is called a \textit{singular interval}, if there exists a  number $\varphi$, so that on the given interval,  a Hamiltonian $H$ has the form
\begin{equation*}
H(x)=h(x)P_{\varphi}, \quad P_{\varphi}=\begin{pmatrix}\cos^2\varphi & \cos\varphi \sin\varphi\\ \cos\varphi\sin\varphi & \sin^2\varphi  \end{pmatrix}
\end{equation*}
for some locally integrable nonnegative function $h$. Note that $P_{\varphi}$ is the projection onto the vector $\left( \begin{smallmatrix} \cos\varphi \\ \sin\varphi \end{smallmatrix} \right)$ and it  is invariant under adding multiples of $\pi$ on $\varphi$.  

To understand the notion of singular intervals, consider (\ref{cs}) on a singular interval $(a,b)$. After multiplying from the left by $J^{-1}=-J$, the system becomes 
$$
{\bf{u}}'(x,z)=-zh(x)JP_{\varphi}{\bf{u}}(x,z).
$$ 
Since the matrices on the right-hand side commute with one another for different values of $x$, the solution is given by 
$$
{\bf{u}}(x,z)=\exp \left( -z\int_a^x h(t) dt J P_{\varphi} \right) {\bf{u}}(a,z).
$$
However, the fact that $P_{\varphi}JP_{\varphi}=0$ (which can be seen either from a direct computation or alternatively from the fact that this matrix is singular, anti-self-adjoint and has real entries) indicates that the series for the exponential terminates and 
\begin{equation}\label{contslnforsi}
{\bf{u}}(x,z)=\left( 1-z\int_a^x h(t)dt J P_{\varphi} \right){\bf{u}}(a,z). 
\end{equation}
In particular, putting 
$H_0=\int_a^b H(x)dx$, we obtain 
$$
J({\bf{u}}(b,z)-{\bf{u}}(a,z))=zH_0{\bf{u}}(a,z).
$$
Hence, on a singular interval, (\ref{cs}) is actually a discrete canonical system in disguise. This treatment of singular intervals will be employed to convert a discrete canonical system to a continuous one in Section \ref{classification}.

\subsubsection{Finite measures and singular intervals}\label{finitemeasures}
As the last ingredient on a canonical system, we  see that a special type of a singular interval is necessary, when its spectral measure is finite, based on \cite{Win4} or \cite{Win2}.  

\begin{Theorem}[Theorem 2.1 \cite{Win4} or Theorem 4.11 \cite{Win2}]
\label{singularintervalwithfm}
The $m$-function of a canonical system has the form 
\begin{equation}\label{mfcnwithfm}
m_H(z)=a+\int_{-\infty}^{\infty} \frac{d\rho(t)}{t-z}
\end{equation} 
with a finite measure $d\rho$ if and only if $0$ is the left end point of a singular interval with $\varphi\neq0$. If $H$ is trace-normed and $(0,l)$ is the maximal singular interval with $\varphi\neq 0$, then the relations $a=\cot\varphi$ and $\int_{-\infty}^{\infty} d\rho(t)=(l \sin^2\varphi)^{-1}$ hold. 
\end{Theorem}
Here a \emph{maximal} singular interval means that there is no essentially larger singular interval containing the given one. In the theorem above, since $(0,l)$ is a maximal singular interval and $0$ is its left end point, any interval $(0,l+\epsilon)$ cannot be a singular interval for $\epsilon>0$. Remark that a singular interval with $\varphi=0$ is related to the $z$-term in Herglotz representation, not to a finite measure. 

By (\ref{HRofJ}), i.e., $a=0$ and $\rho(\R)=1$ in (\ref{mfcnwithfm}), Theorem \ref{singularintervalwithfm} tells us that the canonical systems for Jacobi operators have $(0,1)$ as the maximal singular interval with $\varphi=\pi/2$. For a later use, let us summarize this fact by the following corollary:
\begin{Corollary}\label{Hforjacobi}
Hamiltonians $H$ of the canonical systems corresponding to Jacobi operators are the projection 
$P_{\pi/2}=(\begin{smallmatrix} 0&0\\0&1\end{smallmatrix})$ exactly on
 $(0,1)$.
\end{Corollary}

\section{Classification of Jacobi and Discrete Schr\"odinger eigenvalue equations via canonical systems}
\label{classification}
Based on the same idea in \cite{Hur3}, in this section, we classify trace-normed canonical systems which can be written as the eigenvalue equations of Jacobi or discrete Schr\"odinger operators. For convenience, let us call them   \textit{Jacobi} or \textit{discrete Schr\"odinger canonical systems}, respectively. As talked in the introduction, this new representation is crucial to modify our problem about $m$-functions to the one about canonical systems.  

\subsection{Jacobi canonical systems}\label{Jacobicanonicalsystems}
To find the conditions for Jacobi canonical systems  let's first see that Jacobi eigenvalue equations are actually discrete canonical systems. 

Choose two solutions $c_n$ and $s_n$ to (\ref{jacobieqn}) with zero energy $z=0$, i.e., $a_{n-1}y_{n-1}+a_n y_{n+1}+b_ny_n=0$, such that 
\begin{equation}\label{icforcands}
c_0=s_1=1 \quad\textrm{ and }\quad c_1=s_0=0.
\end{equation} 
As mentioned, $a_0=-1$ in this paper, which will be consistent with (\ref{relationacs}) later. 
It turns out that, if any nonzero $a_0$ were chosen, it would be necessary to select the solution $c_n$ satisfying $a_0c_0=1$ and $c_1=0$, rather than $c_0=1$ and $c_1=0$ (which is in (\ref{icforcands})). 

 For $y_n$ satisfying (\ref{jacobieqn}), put 
\begin{equation}\label{relation}
\begin{pmatrix} y_n \\ y_{n+1} \end{pmatrix}=:
\begin{pmatrix} c_n & s_n \\  c_{n+1} & s_{n+1} \end{pmatrix}
\begin{pmatrix} u^1_n \\ u^2_n \end{pmatrix}
\quad
\left( =T_n {\bf{u}}_n \right)
\end{equation}
to define ${\bf{u}}_n:=(u^1_n,  u^2_n)^t$. Let's denote by $T_n$  the matrix on the right-hand side of (\ref{relation}). Note that $T_n$ has non-zero determinant. Then ${\bf{u}}_n$ satisfies 
\begin{equation}\label{discretecs}
J({\bf{u}}_{n+1}-{\bf{u}}_n)=z
\begin{pmatrix} c_{n+1}^2 & c_{n+1}s_{n+1}\\ c_{n+1}s_{n+1} &s_{n+1}^2 \end{pmatrix}
{\bf{u}}_n
\quad
(=H_{n+1} {\bf{u}}_n).
\end{equation}
(Here the matrix in (\ref{discretecs}) on the right-hand side is denoted by $H_{n+1}$.) 

In order to see \eqref{discretecs} and some underlying principles, we introduce so-called (one-step) transfer matrices $M_n(z)$ by
\begin{equation}\label{M_n(z)}
M_n(z)=
\begin{pmatrix} 0 & 1 \\ -a_{n-1}/a_n & (z-b_n)/a_n \end{pmatrix}.
\end{equation}
Then, as expected from the name of $M_n(z)$, one has $T_n=M_n(0)T_{n-1}$. Solving this for $M_n(0)$ and comparing it to \eqref{M_n(z)} gives two following expressions of $a_n$ and $b_{n+1}$ via the solutions $c_n$ and $s_n$: for all $n\ge 0$, 
\begin{equation}\label{relationacs}
a_n=-\frac1{c_n s_{n+1}-c_{n+1}s_n}
\quad
\left( or
\quad
a_n (c_n s_{n+1}-c_{n+1}s_n)=-1 
\right)
\end{equation}
and 
\begin{equation}\label{relationbcs}
b_{n+1}=a_{n} a_{n+1}(c_{n}s_{n+2}-c_{n+2}s_{n}).
\end{equation} 
Moreover, the equation $T_{n+1}{\bf{u}}_{n+1}=M_{n+1}(z)T_{n}{\bf{u}}_{n}$ holds, since 
\begin{equation*}
T_{n+1}{\bf{u}}_{n+1}=\begin{pmatrix} y_{n+1} \\ y_{n+2} \end{pmatrix}=M_{n+1}(z)\begin{pmatrix} y_n \\ y_{n+1} \end{pmatrix}=M_{n+1}(z)T_{n}{\bf{u}}_{n}.
\end{equation*}
Due to the fact that the quantity, $M_{n+1}(z)-M_{n+1}(0)$, has only one non-zero entry, we see that  
\begin{equation*}
{\bf{u}}_{n+1}-{\bf{u}}_{n}=\frac{z}{a_{n+1}} T^{-1}_{n+1} \begin{pmatrix} 0&0\\0&1 \end{pmatrix} T_n {\bf{u}}_{n},
\end{equation*}
which implies \eqref{discretecs}. 

It has been seen how to convert a Jacobi eigenvalue equation (\ref{jacobieqn}) to the discrete canonical system (\ref{discretecs}).  Remark that there are other ways to alter (\ref{jacobieqn}) to some (\ref{discretecs}). Our way is, however, the one such that their $m$-functions are the same, which will be demonstrated later.\\ 

We now transform (\ref{discretecs}) to a \emph{trace-normed} (continuous) canonical system by doing a change of variables and introducing singular intervals. First do the following change of variables: for $n\ge 1$, 
\begin{equation*}
c_n+is_n=:R_n e^{i\varphi_n},
\end{equation*}
that is, $c_n=R_n \cos \varphi_n$ and $s_n=R_n \sin\varphi_n$. Without a loss of generality, one may assume that $R_n\ge0$. It turns out that $R_n>0$ in our case. Then (\ref{discretecs}) reads 
\begin{equation*}
J({\bf{u}}_{n+1}-{\bf{u}}_n)=z R^2_{n+1} P_{\varphi_{n+1}}{\bf{u}}_n,
\end{equation*}
where 
\begin{equation*}
P_{\varphi_{n}}=\begin{pmatrix} \cos^2 \varphi_{n} & \cos \varphi_{n}\sin \varphi_{n}\\ \cos \varphi_{n}\sin \varphi_{n} & \sin^2 \varphi_{n} \end{pmatrix}.
\end{equation*}
Adding multiples of $\pi$ on $\varphi$ does not change the projection matrix $P_{\varphi}$. Therefore, adapting the idea of \cite{Win3} (Theorem 2.3 therein), let us uniquely normalize $\varphi$ as a non-decreasing, right-continuous step function, such that, for $n\ge 1$, 
\begin{equation}\label{increasingsteps}
\varphi_{n+1}-\varphi_n\in (0,\pi). 
\end{equation}
Make sure that, for the unique representation, steps except the first one should be closed-open intervals (such as $[1,2)$) which are strictly increasing.

To switch a discrete canonical system to a continuous one, the treatment of singular intervals, which was discussed in Section \ref{singularinterval}, is now applied. By expanding or contracting these singular intervals, the system above reveals a \emph{trace-normed} canonical system
\begin{equation}\label{tncswithdet0}
J\frac{d}{dt}{\bf{u}}(t,z)=z 
\begin{pmatrix} \cos^2 \varphi(t) & \cos \varphi(t)\sin \varphi(t)\\ \cos \varphi(t)\sin \varphi(t) & \sin^2 \varphi(t) \end{pmatrix}
{\bf{u}}(t,z).
\end{equation}
Here, by putting $L_0=0$ and 
\begin{equation}\label{defofL}
L_n:=\sum_{k=1}^n R_k^2= \sum_{k=1}^n (c^2_k+s^2_k)\quad \textrm{for } n\ge 1
\end{equation}  
which are the lengths of new singular intervals, the function $\varphi$ should be a non-decreasing and right-continuous step function  on $(0,\infty)$, so that 
 \begin{equation}\label{stepfunction}
\varphi(t) =
\begin{cases}
\varphi_{n+1} \qquad \textrm{on }  [L_n, L_{n+1}) \quad (n\ge1) \\
\varphi_1 (=\pi/2 ) \quad \textrm{on } (0,L_1) \quad (n=0).
\end{cases}
\end{equation}
Here  (\ref{icforcands})  implies that $L_1=1$ and $\varphi_1=\pi/2$, that is, $\varphi(t)=\pi/2$ on $(0,1)$, which is consistent with Corollary \ref{Hforjacobi}.
\\

So far it have been verified that any Jacobi eigenvalue equations can be transformed to some trace-normed canonical systems, such that their Hamiltonian $H$ are projections $P_{\varphi}$ with non-decreasing and right-continuous step functions $\varphi$ satisfying $\varphi(t)=\pi/2$ exactly on $(0,1)$. A typical example of $\varphi$ for some Jacobi canonical system is in the following figure: 
\vspace{0.5cm}

\begin{tikzpicture}[scale=0.65]\label{graphofvarphi}
\draw[->] (0,0) -- (10,0) node[anchor=north] { \large{$t$}};
\draw (0,0) node[anchor=north] {0}
          (2,0) node[anchor=north] {$L_1$(=1)}
          (4.5,0) node[anchor=north] {$L_2$}
          (5.7,0) node[anchor=north] {$L_3$}
          (8,0) node[anchor=north] {$L_4$}
          (0,1.3) node[anchor=east] {$\varphi_1 (=\frac{\pi}2)$}
          (0,2.3) node[anchor=east] {$\varphi_2$}
          (0,3.5) node[anchor=east] {$\varphi_3$}
          (0,4.3) node[anchor=east] {$\varphi_4$};
\draw[->] (0,0) -- (0,6) node[anchor=east] {\large{$\varphi_n$}}; 
\draw[dotted] (2,0) -- (2,6)
                        (4.5,0) -- (4.5,6)
                       (5.7,0) -- (5.7,6)
                       (8,0) -- (8,6);
\draw[very thick]  (0,1.3) -- (2,1.3)
          (2,2.3) -- (4.5,2.3)
          (4.5,3.5) -- (5.7,3.5)
          (5.7,4.3) -- (8,4.3)
          (8,5.3) -- (9.3,5.3);

\draw (-0.03,1.3) -- (0.04,1.3)
          (-0.03,2.3) -- (0.04,2.3)
          (-0.03,3.5) -- (0.04,3.5)
          (-0.03,4.3) -- (0.04,4.3);
\draw (4.2,-1.2) node[anchor=north] {\textbf{Figure A.} \textrm{$\varphi_n$ for a Jacobi canonical system}};
\end{tikzpicture}
\vspace{0.5cm}

Let's finally show that a Jacobi operator and its corresponding Jacobi canonical system share their $m$-functions. First compare their solutions to (\ref{jacobieqn}) and (\ref{tncswithdet0}). More specifically, for $\tilde{y}_n$ a square-summable solution to (\ref{jacobieqn}) near $\infty$, the corresponding solution ${\tilde{\bf{u}}}(t,z)$ via (\ref{relation}) is $P_{\varphi}$-integrable. 
Indeed, 
\begin{eqnarray*}
\sum_{n=1}^{\infty} \tilde{y}_n^* \tilde{y}_n
&=& \sum_{n=1}^{\infty} {\tilde{\bf{u}}}_n^* H_n {\tilde{\bf{u}}}_n\\
&=& \sum_{n=1}^{\infty} {\tilde{\bf{u}}}_n^* R^2_n P_{\varphi_n} {\tilde{\bf{u}}}_n\\
&=& \sum_{n=1}^{\infty} \int_{L_{n-1}}^{L_{n}} {\tilde{\bf{u}}}_n^*  P_{\varphi_n} {\tilde{\bf{u}}}_n \textrm{ } dt\\
&=& \sum_{n=1}^{\infty} \int_{L_{n-1}}^{L_{n}} {\tilde{\bf{u}}}^*(t,z) P_{\varphi_n} {\tilde{\bf{u}}}(t,z)\textrm{ } dt\\
&=& \int_0^{\infty} {\tilde{\bf{u}}}^*(t,z) P_{\varphi} {\tilde{\bf{u}}}(t,z) \textrm{ }dt.
\end{eqnarray*}
Here the first equality holds due to (\ref{relation}). And (\ref{contslnforsi}), combined with the fact that  $P_{\varphi_n}JP_{\varphi_n}=0$, indicates that, on the interval $[L_{n-1}, L_n)$, 
$
 {\tilde{\bf{u}}}(t)^*P_{\varphi_n}{\tilde{\bf{u}}}(t)={\tilde{\bf{u}}}^*(L_n)P_{\varphi_n}{\tilde{\bf{u}}}(L_n), 
$
 which says that the fourth equality holds. By (\ref{mfnforjacobi}), (\ref{mfnforcs}), (\ref{icforcands}) and (\ref{relation}),  we then have that 
\begin{eqnarray*}
m_{J}(z)
&=& -\frac{\tilde{y}_1}{a_0\tilde{y}_0} \quad \left( =\frac{\tilde{y}_1}{\tilde{y}_0} \quad \textrm{since } a_0=-1 \right)\\
&=& \frac{c_1 \tilde{u}^1(0,z)+s_1 \tilde{u}^2(0,z)}{c_0 \tilde{u}^1(0,z)+s_0 \tilde{u}^2(0,z)}\\
&=& \frac{\tilde{u}^2(0,z)}{\tilde{u}^1(0,z)}\\
&=& m_{P_{\varphi}} (z), 
\end{eqnarray*}
which indicates that two $m$-functions are the same.\\ 

Let us summarize what we have discussed so far.
\begin{Theorem}\label{classificationofjcs}
Any Jacobi eigenvalue equation can be written as a canonical system of the form 
\begin{equation*}
J\frac{d}{dt}{\bf{u}}(t,z)=z 
\begin{pmatrix} \cos^2 \varphi(t) & \cos \varphi(t)\sin \varphi(t)\\ \cos \varphi(t)\sin \varphi(t) & \sin^2 \varphi(t) \end{pmatrix}
{\bf{u}}(t,z)
\end{equation*}
where $\varphi$ is a non-decreasing and right-continuous step function on $(0,\infty)$  with $\varphi(t)=\pi/2$ exactly on $(0,1)$, such that the given Jacobi equation and the corresponding canonical system share their $m$-functions.\\
\end{Theorem}

By the same argument but in reverse, it is not hard to show the converse of Theorem \ref{classificationofjcs}: any trace-normed canonical system having such $\varphi$ is a Jacobi eigenvalue equation in disguise. To speak specifically, let's assume that such a $\varphi$ is given, i.e., $\varphi$ satisfies (\ref{stepfunction}) so that $L_0=0$, $\{ L_n \}_{n\ge1}$  is a strictly increasing sequence of positive numbers with $L_1=1$, and $\{\varphi_{n}\}_{n\ge1}$ is a sequence of numbers such that $\varphi_1=\pi/2$ and $\varphi_{n+1}-\varphi_n\in(0,\pi)$. 

Put  
\begin{equation}\label{relationofRandL}
R_{n+1}:=\sqrt{L_{n+1}-L_n} \quad(>0)
\end{equation}
and 
$$c_{n}:=R_{n}\cos\varphi_{n} \quad\textrm{and}\quad s_{n}:=R_{n} \sin\varphi_{n}.$$
Then $\{ {\bf{u}}(n,z) \}$ satisfies (\ref{discretecs}). By (\ref{relation}), the given canonical system reads the Jacobi eigenvalue equation with $a$ and $b$, which are defined through (\ref{relationacs}) and (\ref{relationbcs}). With the same argument above, it can be shown that these two equations have the same $m$-functions.

\subsection{Discrete Schr\"odinger canonical systems}
Based on the previous discussion, let us classify the canonical systems which correspond to discrete Schr\"odinger eigenvalue equations, and then point out how difficult they are to obtain. 

For a discrete Schr\"odinger operators, $a_n\equiv -1$ for all $n\ge 0$. Therefore (\ref{relationacs}) becomes   
\begin{equation*}
c_n s_{n+1}-c_{n+1}s_n=1, 
\end{equation*} 
which, with $R_n$ and $\varphi_n$, expresses 
\begin{equation}\label{condforDSwithRandv}
R_nR_{n+1} \sin(\varphi_{n+1}-\varphi_{n})=1.
\end{equation}
Similar to the case of Jacobi operators, $R_1=1$ (or $L_1=1$) and $\varphi_1=\pi/2$.  Then (\ref{condforDSwithRandv}) implies that $R_2$ should satisfy the condition 
\begin{equation}\label{conditionforR2}
R_2=\csc (\varphi_2-\pi/2),
\end{equation}
which is presented by the blue-dashed curve in Figure B. In particular, $R_2\ge 1$ for any discrete Schr\"odinger operators. A typical second step (i.e., $\varphi_2$ and $R_2$) is the red thick segment in Figure B.\\

\vspace{0.2cm}
\begin{tikzpicture}[scale=0.65]\label{graphofvarphi}
\draw[->] (0,0) -- (13,0) node[anchor=north] { \large{$t$}};
\draw (0,0) node[anchor=north] {0}
          (2.5,0) node[anchor=north] {1}
          (5.1,0) node[anchor=north] {2}
          (7.8,0) node[anchor=north, red] {$L_2=R_2^2+1$}
          (0,1.2) node[anchor=east] {$\varphi_1=\frac{\pi}2$}
          (0,2.4) node[anchor=east] {$\pi$}
          (-1,3.2) node[anchor=east, red] {$\varphi_2$}
          (0,3.6) node[anchor=east] {$\frac{3\pi}2$};
\draw[->] (0,0) -- (0,5) node[anchor=east] {\large{$\varphi_n$}}; 
\draw[dotted]   (2.5,1.2) -- (12.5,1.2)
                         (2.5,2.4) -- (5,2.4)
                         (5.1,0) -- (5.1,2.4)
                       (2.5,3.6) -- (12.5,3.6);
\draw[dotted, red] 
                         (-1,3.2) -- (2.5,3.2)
                        (6.7,0)--(6.7,3.2);
\draw[dashed, blue] (5,2.4)  (5.5,3)  
                       (5.5,3) parabola bend (12.4,3.55) (12.4,3.55)
                      (5.5,1.8) parabola bend (12.4,1.25) (12.4, 1.25);
\draw[dashed, blue] (5.5,3) arc[start angle=110, end angle=250, radius=.62];
\draw[very thick, red] (2.5, 3.2) -- (6.7, 3.2);
\draw[thick]  (0,1.2) -- (2.5,1.2);
\draw (-0.03,1.2) -- (0.04,1.2)
          (-0.03,2.4) -- (0.04,2.4)
          (-0.03,3.5) -- (0.04,3.5)
          (-0.03,4.3) -- (0.04,4.3)
          (5.1, -0.03) -- (5.1, 0.03);
\draw (6.2,-1.2) node[anchor=north] {\textbf{Figure B.} \textrm{$\varphi_n$ for discrete Schr\"odinger canonical systems}};
\end{tikzpicture}\\

Different from Jacobi canonical systems, (\ref{condforDSwithRandv}) says that the length of the next step is determined by $\varphi$ when knowing the previous step. This is because there is no freedom to choose $a_n$ for discrete Schr\"odinger operators. This shows that it is very rare to obtain discrete Schr\"odinger canonical systems. 

With the similar process before, we present the classification of the discrete Schr\"odinger  canonical systems  by the following theorem: 
\begin{Theorem}\label{classificationofdscs}
Assume that $L_0=0$ and  $\{ L_n \}_{n\ge1}$ is a strictly increasing sequence of positive numbers with $L_1=1$. Put $R_1=1$ and, for $n\ge1$,  
$$R_{n+1}:=\sqrt{L_{n+1}-L_n}$$ 
(which is (\ref{relationofRandL})). Then any canonical system of the form 
\begin{equation*}
J\frac{d}{dt}{\bf{u}}(t,z)=z 
\begin{pmatrix} \cos^2 \varphi(t) & \cos \varphi(t)\sin \varphi(t)\\ \cos \varphi(t)\sin \varphi(t) & \sin^2 \varphi(t) \end{pmatrix}
{\bf{u}}(t,z)
\end{equation*}
with $\varphi$ satisfying  (\ref{stepfunction}), (\ref{increasingsteps}) and (\ref{condforDSwithRandv}) can be written as a discrete Schr\"odinger eigenvalue equation
$$
-y_{n+1}-y_{n-1}+b_ny_n=zy_n, 
$$
where  
\begin{equation}\label{formulaforp}
b_{n+1}=R_nR_{n+2} \sin(\varphi_{n+2}-\varphi_n).
\end{equation}
Moreover, their $m$-functions are the same. 
\end{Theorem}
Remark that, compared to the continuous case, (\ref{condforDSwithRandv}) and (\ref{formulaforp}) correspond to (20) and (25) in \cite{Hur3}, respectively.

\section{Non-density of discrete Schr\"odinger $m$-functions}\label{example}
In this section the $m$-functions for discrete Schr\"odinger operators are shown not to be dense on those for Jacobi operators. As mentioned in the introduction, this result is opposite to the density of Schr\"odinger $m$-functions on all Herglotz functions in \cite{Hur3} for the continuous setting. 

To see this, the one-to-one correspondence between Herglotz functions and canonical systems is now utilized to deal with the problem on canonical systems. Based on both  Theorem \ref{classificationofjcs} and Theorem \ref{classificationofdscs} (which are the representations of Jacobi and discrete Schr\"odinger canonical systems respectively), it is enough to construct some Jacobi canonical system whose $m$-function cannot be approximated by those of discrete Schr\"odinger operators.

For this, a topology on canonical systems is required, such that it accords with the local uniform convergence on Herglotz functions. Besides de Branges' own work \cite{deB} for this (see also Lemma 2.1 in \cite{Win4}), we adapt the following convergence on canonical systems in \cite{Hur3}: 

\begin{Lemma}[Proposition 5.1 in \cite{Hur3}]\label{topologyoncs}
The convergence  $m_{H_n}(z)\to m_H(z)$, $n\to\infty$, holds locally uniformly for $z\in\C^+$ if and only if it holds 
\begin{equation}
\label{weak*conv}
\int_0^{\infty}{\boldsymbol{f}}^*H_n {\boldsymbol{f}}\to\int_0^{\infty}{\boldsymbol{f}}^*H {\boldsymbol{f}}
\end{equation}
for all continuous functions ${\boldsymbol{f}}=(f^1,f^2)^t$ with compact support of $[0,\infty)$.
\end{Lemma}
Here, let us say that \emph{$H_n$ converges to $H$ weak-$\ast$}, when (\ref{weak*conv}) holds. Note that the trace-normed condition, $\textrm{Tr }H(t)=1$, is preserved in the limiting process, which indicates that $H(t)dt$ is absolutely continuous with respect to the Lebesgue measure. Since there is no point spectrum for these measures,  any vectors containing characteristic functions, such as $(\chi_{[1,2)}(t),0)^t$, can be test functions in Lemma \ref{topologyoncs}. In other words, the weak-$\ast$ convergence may be operated with any characteristic functions. \\ 

We are ready to see the non-density of the $m$-functions for discrete Schr\"odinger operators on all those for Jacobi operators. 
\begin{Theorem}\label{nodensity}
There is a Jacobi operator whose $m$-function cannot be approximated by the $m$-functions for discrete Schr\"odinger operators in the sense of the local uniform convergence. 
\end{Theorem}
\begin{proof}
It is enough to give some Jacobi operator whose $m$-function cannot be approachable by the $m$-functions of discrete Schr\"odinger operators. By de Branges theory this is equivalent to find a Jacobi canonical system which cannot be close by the discrete Schr\"odinger canonical systems in the sense of the topology induced by Lemma \ref{topologyoncs}, as talked.

Based on this idea, consider any Jacobi canonical system satisfying
\begin{equation*}
\varphi^\mathcal{J} (t)=\begin{cases} 3\pi/4  \quad (=\varphi^\mathcal{J}_2) \quad \textrm{on } [1,3/2) \\ \pi \quad (=\varphi^\mathcal{J}_3) \quad \textrm{on } [3/2,2) \end{cases}.
\end{equation*}
(Here the superscript $\mathcal{J}$ is for Jacobi canonical systems.) In other words, its second and third steps look like ones on the following figure: 

\begin{tikzpicture}[scale=0.65]\label{graphofvarphi}
\draw[->] (0,0) -- (13,0) node[anchor=north] { \large{$t$}};
\draw (0,0) node[anchor=north] {0}
          (2.5,0) node[anchor=north] {1}
          (5.1,0) node[anchor=north] {2}
           (0,1.2) node[anchor=east] {$\varphi^\mathcal{J}_1=\frac{\pi}2$}
          (0,2.4) node[anchor=east,red] {$\varphi^\mathcal{J}_3=\pi$}
          (-2.1,1.85) node[anchor=east, red] {$\varphi^\mathcal{J}_2$}
          (0,3.6) node[anchor=east] {$\frac{3\pi}2$};
\draw[->] (0,0) -- (0,5) node[anchor=east] {\large{$\varphi^\mathcal{J}_n$}}; 
\draw[dotted]   (2.5,1.2) -- (12.5,1.2)
                         (5.1,0) -- (5.1,2.4)
                       (2.5,3.6) -- (12.5,3.6);
\draw[dotted, red]                           (0,2.4) -- (5,2.4)
                         (-2.1,1.85) -- (2.5,1.85);
\draw[dashed, blue] (5,2.4)  (5.5,3)  
                       (5.5,3) parabola bend (12.4,3.55) (12.4,3.55)
                      (5.5,1.8) parabola bend (12.4,1.25) (12.4, 1.25);
\draw[dashed, blue] (5.5,3) arc[start angle=110, end angle=250, radius=.62];
\draw[very thick, red] (3.9, 2.4) -- (5.1, 2.4)
                    (2.5, 1.85) -- (3.9, 1.85);
\draw[thick]  (0,1.2) -- (2.5,1.2);
\draw (-0.03,1.2) -- (0.04,1.2)
          (-0.03,2.4) -- (0.04,2.4)
          (-0.03,3.5) -- (0.04,3.5)
          (-0.03,4.3) -- (0.04,4.3)
          (5.1, -0.03) -- (5.1, 0.03);

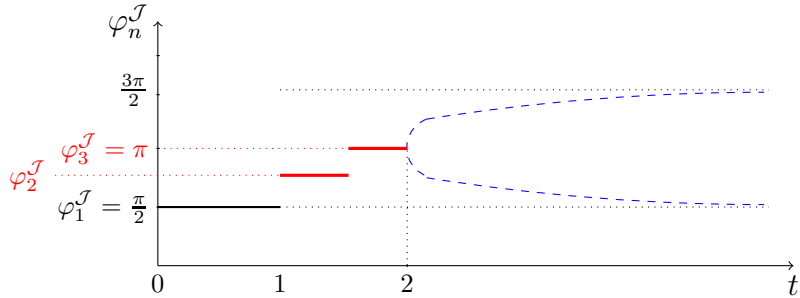
\draw (6.2,-1.2) node[anchor=north] {\textbf{Figure C.} \textrm{Some $\varphi^\mathcal{J}_n$ far from discrete Schr\"odinger canonical systems}};
\end{tikzpicture}\\

We now see that  discrete Schr\"odinger canonical systems cannot approach the Jacobi canonical system. Recall first that, due to (\ref{conditionforR2}), any $\varphi^{DS}$ for discrete Schr\"odinger canonical systems should be constant at least on $[1,2)$. 

Let us focus on the interval $[1,2)$. Since any constant function cannot be close to two different steps at the same time in the $L_1$-sense, Hamiltonians of discrete  Schr\"odinger canonical systems cannot converge in weak-$\ast$ to the one for the given Jacobi canonical system. More precisely, put $\varphi^{DS}_2=3\pi/4$ which is for the second step of discrete Schr\"odinger canonical system. Take a look at  the following figure to see the situation. 

\vspace{0.6cm}
\begin{tikzpicture}[scale=0.65]\label{graphofvarphi}
\draw           (2.5,0) node[anchor=north] {1}
          (3.9,0) node[anchor=north] {3/2}
          (5.1,0) node[anchor=north] {2}
          (0.5,2.2) node[anchor=east] {$\varphi^\mathcal{J}_3=\pi$}
          (0,1.4) node[anchor=east] {$\varphi^J_2=\frac{3\pi}4$}
          (0.5,3.6) node[anchor=east] {$\frac{3\pi}2$};
\draw (7.0,2.0) node[anchor=west, red] {$\varphi^{DS}_2$};
\draw[dotted]   (3.9, 2.2)--(3.9,0) 
                         (2.5,3.6) -- (5.1,3.6);
\draw[dotted]   (0,1.4) -- (2.5,1.4)
                         (0.5,2.2) -- (3.9,2.2);
\draw[very thick] (3.9, 2.2) -- (5.1, 2.2)
                    (2.5, 1.4) -- (3.9, 1.4);
\draw[dotted]  (2.5,.6)--(5.1,.6) ;
\draw           (5.1, -0.03) -- (5.1, 0.03);
\draw[very thick, red] (2.5, 1.8) -- (5.1,1.8);
\draw[dotted, red] (5.1,1.8) -- (7.0,1.8);
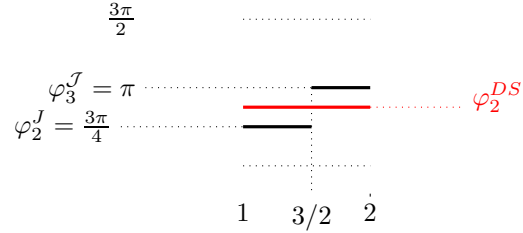
\draw (2.0,-1.2) node[anchor=north] {\textbf{Figure D.} \textrm{Failing $L_1$- approximation by the second step}};
\end{tikzpicture}\\

By direct computation, see that, for two vectors ${\boldsymbol{f}}=(\chi_{[1,3/2)},0)^t$ and ${\boldsymbol{g}}=(\chi_{[3/2,2)},0)^t$ which are used as test functions in Lemma \ref{topologyoncs},  
\begin{equation*}
\max \left\{ \left| \int_{\R} {\boldsymbol{f}}^*  \left(P_{\varphi^J}-P_{\varphi^{DS}} \right)  {\boldsymbol{f}} \right|,  
\left| \int_{\R} {\boldsymbol{g}}^*  \left(P_{\varphi^J}-P_{\varphi^{DS}} \right)  {\boldsymbol{g}} \right|\right\}=1/2.
\end{equation*}
This is because 
$
(1,0) (\begin{smallmatrix} h_{1} & h_2 \\ h_2 & h_3 \end{smallmatrix}) \left( \begin{smallmatrix} 1 \\0 \end{smallmatrix}\right) = h_1. 
$
It is easy to show that this is the minimum error (see Figure D), that is, for any $\varphi^{DS}$ for discrete Schr\"odinger canonical system, we have that 
\begin{equation*}
\max \left\{ \left| \int_{\R} {\boldsymbol{f}}^*  \left(P_{\varphi^J}-P_{\varphi^{DS}} \right)  {\boldsymbol{f}} \right|,  
\left| \int_{\R} {\boldsymbol{g}}^*  \left(P_{\varphi^J}-P_{\varphi^{DS}} \right)  {\boldsymbol{g}} \right|\right\}\ge1/2. 
\end{equation*}
All this  implies that the given Jacobi canonical system cannot be approximated by discrete Schr\"odinger canonical systems, as desired. 
\end{proof}

 As a next project, we may think of how large or small the set of either continuous or discrete Schr\"odinger operators is. This will reveal the relations between Schr\"odinger operators and the larger space, the set of either  canonical systems or Jacobi operators, which tells us more about their inverse spectral theories.\\

\textit{Acknowledgement.} The author is grateful to the referee for valuable discussion. This research was supported by Basic Science Research Program through the National Research Foundation of Korea (NRF) funded by the Ministry of Education (NRF-2016R1D1A1B03931764).\\ 


\begin{thebibliography}{10000}
\bibitem{RemLN} C. Remling, Lecture note Chapter 13 for functional analysis, which can be downloaded from  the homepage www2.math.ou.edu/$\sim$cremling.
\end{thebibliography}

\end{document}